\documentclass[paper=a4, fontsize=11pt]{scrartcl} 

\usepackage[T1]{fontenc} 
\usepackage{fourier} 
\usepackage[english]{babel} 
\usepackage{amsmath,amsfonts,amsthm} 
\usepackage{amsmath, calligra, mathrsfs}
\usepackage{amssymb}
\usepackage{mathtools}
\usepackage{hyperref}
\usepackage{mathdots}
\usepackage{array}
\usepackage[mathscr]{euscript}
\usepackage{verbatim}
\usepackage[dvipsnames]{xcolor}
\usepackage{mdframed} 
\usepackage{soulutf8}
\usepackage{stmaryrd}
\usepackage{tikz-cd}
\usepackage{hyperref}
\usepackage{lipsum} 

\usepackage{sectsty} 
\allsectionsfont{\centering \normalfont\scshape} 

\usepackage{fancyhdr} 
\usepackage{xurl}
\newcommand*{\sheafhom}{\mathcal{H}\kern -.5pt om}
\numberwithin{equation}{section} 
\numberwithin{figure}{section} 
\numberwithin{table}{section} 

\newtheorem{thm}{Theorem}[section]
\newtheorem{cor}[thm]{Corollary}
\newtheorem{prop}[thm]{Proposition}

\theoremstyle{definition}
\newtheorem{defn}[thm]{Definition}

\theoremstyle{remark}
\newtheorem{rem}[thm]{Remark}

\DeclareMathOperator{\St}{st}

\DeclareMathOperator{\lk}{lk}

\DeclareMathOperator{\Ast}{ast}

\newcommand{\horrule}[1]{\rule{\linewidth}{#1}} 

\title{	
	\normalfont \normalsize 
	\textsc{} \\ [25pt] 
	\horrule{0.5pt} \\[0.4cm] 
	\huge Recursive properties of Cohen--Macaulay flag simplicial complexes and Lefschetz decompositions from $f$-vectors
	
	\horrule{2pt} \\[0.5cm] 
}

  \author{Soohyun Park \\ \href{mailto:soohyun.park@mail.huji.ac.il}{soohyun.park@mail.huji.ac.il} } 

\date{\normalsize October 23, 2024} 

\begin{document}
	
	\maketitle

	\begin{abstract}
		\noindent Most applications of the hard Lefschetz theorem related to combinatorial properties of simplicial complexes involve their $h$-vectors. In the context of positivity properties involving $h$-vectors of flag spheres, $f$-vectors with a Lefschetz-type ``Boolean'' decomposition have been studied. In this note, we explore families of flag simplicial complexes where we can see this Boolean decomposition explicitly in terms of transformations connecting different simplicial complexes in this family. Note that we will take complexes in a given dimension to be PL homeomorphic to each other. In particular, the existence of a Boolean decomposition patched from local parts can be phrased in terms of a certain map formally satisfying an analogue of the hard Lefschetz theorem. The map is given by the composition of a double suspension with a ``net single edge subdivision''. Here, the former contributes to the Boolean part and the latter contributes to the disjoint non-Boolean part. The fact that the simplicial complex with the given $f$-vector can be taken to be balanced suggests algebraic versions of maps connected to these decompositions.
	\end{abstract}

	\section*{Introduction}

	The Hard Lefschetz theorem has been a source of many interesting combinatorial results related to $h$-vectors of simplicial polytopes (e.g. p. 77 -- 78 of \cite{St}). In general, one usually works with $h$-vectors of simplicial complexes while making use of geometric results. However, there are also $f$-vectors of simplicial complexes which seem to exhibit this property which are studied in the context of $h$-vectors of flag spheres (e.g. \cite{NP}, \cite{NPT}, \cite{Pbalfveclef}). To be more specific, this involves the following decomposition: \\
	
	\begin{defn} (Proposition 6.2 on p. 1377 of \cite{NPT}) \\ 
		A \textbf{Boolean decomposition} of a $(d - 1)$-dimensional simplicial complex $\Gamma$ is a decomposition \[ \Gamma = \{ F \cup G : F \in S, G \in 2^{[d - 2|F|]} \} \] where $S \le \Gamma$ is the subcomplex consisting of faces of $\Gamma$ disjoint from $[d]$. We will refer to $F \in S$ as the initial part and $G \in 2^{[d - 2|F|]}$ as the Boolean part. \\
	\end{defn}
	
	Moreover, the degrees \emph{not} being halved indicates a structure that is more similar to a Lefschetz decomposition from a geometric source. \\
	
	In this note, we would like to further explore other examples where $f$-vectors equal to $h$-vectors of the given simplicial complexes have such a structure. To this end, we start with recursive structures in a more general class of flag simplicial complexes where those of a given dimension are PL homeomorphic to each other. In addition, we also consider how the Boolean decomposition appears in a more constructive way. For example, we use results older than those used to show the existence of simplicial complexes whose $f$-vectors are equal to $h$-vectors of given flag simplicial complexes (Corollary 2.3 on p. 474 of \cite{CCV}, Theorem on p. 23 of \cite{BFS}) to carry out explicit constructions. This includes earlier work in \cite{CV}. In the course of doing this, we end up studying local-global behavior of the decomposition. When the starting simplicial complex is a pseudomanifold, this amounts to comparisons with Boolean decompositions of codimension 2 flag spheres. This is summarized below: \\
	
	\begin{thm}
		Let $\Delta$ be a flag Cohen--Macaulay simplicial complex and $\Gamma$ be a simplicial complex such that $h(\Delta) = f(\Gamma)$. Define $\Gamma_e$ analogously for $\lk_\Delta(e)$. \\
		\begin{enumerate}
			\item (Proposition \ref{subconchange}) Define $\Gamma'$ for the subdivision $\Delta'$ of $\Delta$ with respect to the edge $e$. If $\Gamma$ and $\Gamma_e$ have Boolean decompositions, so does $\Gamma'$. The subdivision adds a point to the initial non-Boolean part. \\
			
			\item (Proposition \ref{pltransbool}) Define $\widetilde{\Gamma}$ for the contraction $\widetilde{\Delta}$ of an edge $e \in \Delta$. Then, $\widetilde{\Gamma}$ has a Boolean decomposition if and only if $\widetilde{\Gamma} = \Ast_\Gamma(u)$ in a way compatible with the Boolean decomposition of $\Gamma$ (i.e. removing a vertex from the initial part $S$ disjoint from the Boolean part). For example, this is the case when a double suspension of $\lk_\Delta(e)$ is contained in $\Delta$. Note that $f_0(\widetilde{\Gamma}) = f_0(\Gamma) - 1$. \\
			
			\item (Corollary \ref{locglobbool})  Consider a collection $\mathcal{C}$ of flag Cohen--Macaualy simplicial complexes $\Delta$ satisfying the following properties: \\
			
			\begin{itemize}
				\item The collection $\mathcal{C}$ is closed under (double) suspensions. \\
				
				\item The simplicial complexes of a given dimension in $\mathcal{C}$ are PL homeomorphic to each other. \\
			\end{itemize}
			
			Suppose that $\Gamma_e$ has a Boolean decomposition for each edge $e \in \Delta$. These come from codimension 2 flag spheres when we start with a pseudomanifold. \\
			
			Then, $\Gamma$ has a Boolean decomposition if and only if the composition of a double suspension and a ``net single edge subdivision'' (truncated after degree $d$) yields a map $L$ formally satisfying the linear properties in the Hard Lefschetz theorem. \\

		\end{enumerate}
	\end{thm}
	
	The constructions we consider also suggest potential algebraic maps related to the combinatorial analogue of the Lefschetz maps considered above (Remark \ref{algdisc}).

	\section{Boolean decompositions, transformations, and local-global properties via Lefschetz maps}
	
	We would like to study the local-global behavior of decompositions of simplicial complexes whose $f$-vectors give $h$-vectors of flag Cohen--Macaulay simplicial complexes $\Delta$. Before specifying the families which we are studying, we define the types of decompositions that we are interested in. \\
	
	\begin{defn} (Proposition 6.2 on p. 1377 of \cite{NPT}) \\ \label{booldef}
		A \textbf{Boolean decomposition} of a $(d - 1)$-dimensional simplicial complex $\Gamma$ is a decomposition \[ \Gamma = \{ F \cup G : F \in S, G \in 2^{[d - 2|F|]} \} \] where $S \le \Gamma$ is the subcomplex consisting of faces of $\Gamma$ disjoint from $[d]$. We will refer to $F \in S$ as the initial part and $G \in 2^{[d - 2|F|]}$ as the Boolean part. \\
	\end{defn}

	In particular, we would like to consider collections $\mathcal{C}$ of simplicial complexes satisfying the following properties via local properties (which come from codimension 2 flag spheres when we start with a pseudomanifold): \\
	
	\begin{itemize}
		\item The collection $\mathcal{C}$ is closed under (double) suspensions. \\
		
		\item The simplicial complexes of a given dimension in $\mathcal{C}$ are PL homeomorphic to each other. \\
	\end{itemize}
	
	The first property comes from the structure of the Boolean decomposition from Definition \ref{booldef}. Each suspension multiplies the $h$-polynomial of the starting simplicial complex by $t + 1$ (e.g. see p. 516 of \cite{NP}). On the $f$-vector side, this is like a double coning of $\Gamma$. We move between simplicial complexes satisfying the second property using the following result: \\

	\begin{thm} (Lutz--Nevo, Theorem 1.2 on p. 70, 77 -- 78 of \cite{LN}) \\
		Two flag simplicial complexes $\Delta_1$ and $\Delta_2$ are PL homeomorphic if and only if they can be connected by a sequences of edge subdivisions and their inverses such that all the complexes in the sequences are flag. Note that the edges contracted must be admissible (i.e. not contained in an induced 4-cycle). \\
	\end{thm}
	
	Although there are general results of Caviglia--Constantinescu--Varbaro (Corollary 2.3 on p. 474 of \cite{CCV}) and Bj\"orner--Frankl--Stanley (Theorem on p. 23 of \cite{BFS}) which imply the existence of such simplicial complexes when the given simplicial complex $\Delta$ is flag and Cohen--Macaulay, we will use constructions used to study special cases from earlier work of Constantinescu--Varbaro \cite{CV} constructing concrete examples of the simplicial complexes involved and its recursive behavior. We use these tools to look at how $\Gamma$ such that $f(\Gamma) = h(\Delta)$ can vary among PL homeomorphic simplicial complexes $\Delta$. \\
	
	\begin{defn} (Definition 5.3.4 on p. 232 of \cite{BH}, p. 60 of \cite{St}) \\
		Given a simplicial complex $\Delta$ and a face $F \in \Delta$, its link of $\Delta$ over $F$ is defined as \[ \lk_\Delta(F) \coloneq \{ A \in \Delta : A \cup F \in \Delta, A \cap F = \emptyset \}. \]
	\end{defn}

	\begin{prop} \label{subconchange}
		Let $\Delta$ be a flag Cohen--Macaulay simplicial complex and $\Gamma$ be a simplicial complex such that $h(\Delta) = f(\Gamma)$ (Corollary 2.3 on p. 474 of \cite{CCV} and Theorem 1 on p. 23 of \cite{BFS}). Define $\Gamma_e$ analogously for $\lk_\Delta(e)$. \\
		
		\begin{enumerate}
			\item If $\Delta'$ is the (stellar) subdivision of $\Delta$ with respect to an edge $e \in \Delta$, we have that $h(\Delta') = f(\Gamma')$, where \[ \Gamma' = \Gamma *_{\Gamma_e} u = \Gamma \cup (\Gamma_e * u) \] with $\Gamma_e$ a simplicial complex such that $h(\lk_\Delta(e)) = f(\Gamma_e)$ considered as a subcomplex of $\Gamma$. \\
			
			\item Suppose that $\widetilde{\Delta}$ is the contraction of an edge $e \in \Delta$. Then, we have that $h(\widetilde{\Delta}) = f(\widetilde{\Gamma})$ with \[ \widetilde{\Gamma} = \Gamma - (\Gamma_e + u), \] where the removed term denotes faces formed by adding a new ``outside'' vertex $u$ added to faces of $\Gamma_e$. In this construction, we have $\widetilde{\Gamma} = \Ast_\Gamma(u)$ and $\Gamma_e = \lk_\Gamma(u)$. Note that $u$ does \emph{not} belong to the unique facet of $\Gamma$ if $h_d(\Delta) = 1$ and $h_d(\widetilde{\Delta}) = 1$. \\
			
		\end{enumerate}
	\end{prop}
	
	\begin{proof}
		\begin{enumerate}
			\item Recall (e.g. from p. 36 of \cite{Athgam}) that \[ h_{\Delta'}(t) = h_\Delta(t) + t h_{\lk_\Delta(e)}(t), \] where $\Delta'$ denotes the (stellar) subdivision of $\Delta$ with respect to the edge $e \in \Delta$. In particular, this means that \[ h_i(\Delta') = h_i(\Delta) + h_{i - 1}(\lk_\Delta(e)) \] and \[ f_{i - 1}(\Gamma') = f_{i - 1}(\Gamma) + f_{i - 2}(\Gamma_e). \] Following work of Constantinescu--Varbaro \cite{CV}, this means that \[ \Gamma' = \Gamma *_{\Gamma_e} u = \Gamma \cup (\Gamma_e * u) \] for some new disjoint vertex $u$. \\
			
			\item Now consider edge contractions. In other words, we are looking at how to express $\Gamma$ in terms of $\Gamma'$. Since \[ h_{\Delta'}(t) = h_\Delta(t) + t h_{\lk_\Delta(e)}(t), \] we have \[ h_\Delta(t) = h_{\Delta'}(t) - t h_{\lk_\Delta(e)}(t). \]
			
			Note that $\lk_\Delta(e) \cong \lk_{\Delta'}(p, v) = \lk_{\Delta'}(v, q)$, where $v \in V(\Delta')$ is the subdividing vertex of the edge $e \in \Delta$. \\
			
			Suppose that $\Gamma'$ itself has a Boolean decomposition. Assuming that $\Gamma_e$ also has a Boolean decomposition by induction on dimension, the construction $\Gamma' = \Gamma *_{\Gamma_e} u = \Gamma \cup (\Gamma_e * u)$ above implies that \[  \Gamma' = \Gamma \cup \{ (F_e \cup u) \cup G_e  : F_e \in S_e, G_e \in 2^{[d - 2|F_e| - 2]} \} \] and \[ \Gamma = \Gamma' \setminus \{ (F_e \cup u) \cup G_e  : F_e \in S_e, G_e \in 2^{[d - 2|F_e| - 2]} \}. \]
			
			This amounts to removing the faces that contain $u$, which is a vertex in $S'$. \\

			We can also approach this more directly. Given a flag sphere $\Delta$, let $\widetilde{\Delta}$ be the contraction of $\Delta$ with respect to an edge $e = (a, b)$ obtained by setting $a = b$. Note that we do not take arbitrary edges (which don't necessarily preserve flagness), but ones that are admissible (i.e. not contained in any induced 4-cycle -- see p. 79 -- 80 of \cite{LN}). In addition, we note that boundaries of cross polytopes yield minimal $h$-vectors for doubly Cohen--Macaulay simplicial complexes (Theorem 1.3 on p. 19 of \cite{Athsom}). We can record the changes in the faces below:
			
			\begin{itemize}
				\item Faces that do \emph{not} contain $a$ or $b$ are unaffected.
				
				\item Faces that contain $e$ are removed.
				
				\item  Faces that only contain $a$ or only contain $b$ now have a common vertex. If they come from adding $a$ to elements of $\lk_\Delta(a) \setminus (\lk_\Delta(a) \cap \lk_\Delta(b)) = \lk_\Delta(a) \setminus \lk_\Delta(e)$, they now have a vertex in common with those coming from adding $b$ to $\lk_\Delta(b) \setminus (\lk_\Delta(a) \cap \lk_\Delta(b)) = \lk_\Delta(b) \setminus \lk_\Delta(e)$. However, such pairs of faces remain distinct.
				
				\item  In other words, the change in the number of faces come from those that are no longer distinct or no longer exist after setting $a = b$. This means faces of the form $F \cup a$ or $F \cup b$ for $F \in \lk_\Delta(e)$ (which are merged) or $F \cup e$ for $F \in \lk_\Delta(e)$ (which are removed). \\ 
			\end{itemize}
			
			Putting this together, we have
			
			\begin{align*}
				f_\Delta(t) &= f_{\St_\Delta(e)}(t) + f_{\Delta \setminus \St_\Delta(e)}(t) \\
				&= (t + 1)^2 f_{\lk_\Delta(e)}(t) + f_{\Delta \setminus \St_\Delta(e)}(t) \\
				&= t^2 f_{\lk_\Delta(e)}(t) + 2t f_{\lk_\Delta(e)}(t) + f_{\lk_\Delta(e)}(t) + f_{\Delta \setminus \St_\Delta(e)}(t)
			\end{align*}
			
			and
			
			\begin{align*}
				f_{\widetilde{\Delta}}(t) &= t f_{\lk_\Delta(e)}(t) + f_{\lk_\Delta(e)}(t) + f_{\Delta \setminus \St_\Delta(e)}(t),
			\end{align*}

			which implies that
			
			\begin{align*}
				f_{\widetilde{\Delta}}(t) &= f_\Delta(t) - (t^2 + t) f_{\lk_\Delta(e)}(t) \\
				&= f_\Delta(t) - t(1 + t) f_{\lk_\Delta(e)}(t).
			\end{align*}
			
			Using the identity \[ \sum_{i = 0}^d f_{i - 1} (t - 1)^{d - i} = \sum_{k = 0}^d h_k t^{d - k}, \] we can relate their $h$-polynomials after substituting in $t = \frac{1}{u - 1}$ and then $w = \frac{1}{u}$:
			
			\begin{align*}
				f_{\widetilde{\Delta}} \left( \frac{1}{u - 1} \right) &= f_\Delta \left( \frac{1}{u - 1} \right) - \frac{1}{u - 1} \cdot \frac{u}{u - 1} f_{\lk_\Delta(e)} \left( \frac{1}{u - 1} \right) \\
				\Longrightarrow (u - 1)^d f_{\widetilde{\Delta}} \left( \frac{1}{u - 1} \right) &= (u - 1)^d f_\Delta \left( \frac{1}{u - 1} \right) - u (u - 1)^{d - 2} f_{\lk_\Delta(e)} \left( \frac{1}{u - 1} \right) \\
				\Longrightarrow u^d h_{\widetilde{\Delta}} \left( \frac{1}{u} \right) &= u^d h_\Delta \left( \frac{1}{u} \right) - u \cdot u^{d - 2} h_{\lk_\Delta(e)} \left( \frac{1}{u} \right) \\
				\Longrightarrow h_{\widetilde{\Delta}} \left( \frac{1}{u} \right) &= h_\Delta \left( \frac{1}{u} \right) - \frac{1}{u} h_{\lk_\Delta(e)} \left( \frac{1}{u} \right) \\
				\Longrightarrow h_{\widetilde{\Delta}}(w) &= h_\Delta(w) - w h_{\lk_\Delta(e)}(w).
			\end{align*}
			
			In other words, we have \[ h_i(\widetilde{\Delta}) = h_i(\Delta) - h_{i - 1}(\lk_\Delta(e)). \]
			
			For simplicial complexes $\widetilde{\Gamma}$, $\Gamma$, and $\Gamma_e$ such that $f(\widetilde{\Gamma}) = h(\widetilde{\Delta})$, $f(\Gamma) = h(\Delta)$, and $f(\Gamma_e) = h(\lk_\Delta(e))$, this means that \[ f_{i - 1}(\widetilde{\Gamma}) = f_{i - 1}(\Gamma) - f_{i - 2}(\Gamma_e). \]
			
			Note that $\widetilde{\Delta}$ still has the property of being a flag sphere and $\lk_\Delta(e) \le \widetilde{\Delta}$ is a subcomplex which is unchanged after the edge contraction. Note that both of these simplicial complexes are Cohen--Macaulay since they are both (flag) Cohen--Macaulay simplicial complexes, which are preserved by taking links of faces and edge contractions. This means that $h_i(\lk_\Delta(e)) \le h_i(\widetilde{\Delta})$ for all $i$ (Theorem 9.1 on p. 126 -- 127 of \cite{St}). Using compression complexes, this means that we can take $\Gamma_e \le \widetilde{\Gamma}$ as a subcomplex. We can set \[ \Gamma = \widetilde{\Gamma} *_{\Gamma_e} u = \widetilde{\Gamma} \cup (\Gamma_e * u) \] for some ``new'' vertex $u$ of $\Gamma$ not belonging to $\widetilde{\Gamma}$ or $\Gamma_e$ (e.g. see p. 96 -- 97  of \cite{CV}). Here, we have $\widetilde{\Gamma} = \Ast_\Gamma(u)$ and $\Gamma_e = \lk_\Gamma(u)$. Note that $f_0(\widetilde{\Gamma}) = f_0(\Gamma) - 1$ and that the missing vertex cannot be in $[d]$ since $\dim \Gamma = \dim \widetilde{\Gamma}$. Equivalently, we have \[ \widetilde{\Gamma} = \Gamma - (\Gamma_e + u) \] with the second part meaning faces of $\Gamma$ formed by adding $u$ to those of $\Gamma_e$ since $\Gamma_e \le \widetilde{\Gamma}$ is a subcomplex. 
		\end{enumerate}
	\end{proof}
	
	We can express the condition for the contraction to also induce a Boolean decomposition in a simpler way. \\
	
	\begin{prop} \label{pltransbool}
		Use the notation from Proposition \ref{subconchange}. Suppose that $\Gamma$ and $\Gamma_e$ both have Boolean decompositions. \\
		
		\begin{enumerate}
			\item $\Gamma'$ also has a Boolean decomposition. \\
			
			\item  $\widetilde{\Gamma}$ has a Boolean decomposition if and only if $\widetilde{\Gamma} = \Ast_\Gamma(u)$ in a way compatible with the Boolean decomposition of $\Gamma$ (i.e. removing a vertex from the initial part $S$ disjoint from the Boolean part). For example, this is the case when a double suspension of $\lk_\Delta(e)$ is contained in $\Delta$. Note that $f_0(\widetilde{\Gamma}) = f_0(\Gamma) - 1$. \\
		\end{enumerate}
		
	\end{prop}

	\begin{proof}
		\begin{enumerate}
			\item Add the vertex $u$ added to the subcomplex $S$ disjoint from Boolean part. If there are Boolean decompositions \[ \Gamma = \{ F \cup G : F \in S, G \in 2^{[d - 2|F|]} \} \] and \[ \Gamma_e = \{ F_e \cup G_e  : F_e \in S_e, G_e \in 2^{[d - 2|F_e| - 2]} \}, \] we have that \[ \Gamma_e * u = \Gamma_e \cup \{ (F_e \cup u) \cup G_e  : F_e \in S_e, G_e \in 2^{[d - 2|F_e| - 2]} \}. \] The latter set of terms give something that already is compatible with the Boolean decomposition in $\Gamma$ as a whole since adding an additional vertex to the the part of the face belonging to the initial component $S$ decreases the number of ``available'' Boolean parts by 2. In particular, this implies that \[ \Gamma' = \Gamma \cup \{ (F_e \cup u) \cup G_e  : F_e \in S_e, G_e \in 2^{[d - 2|F_e| - 2]} \} \] since $\Gamma_e \le \Gamma$ is a subcomplex. \\
			
			\item In order to have a Boolean decomposition, the only thing that can be changed is a dimension or a restriction on $S$ (e.g. on the vertex set of $S$). Since there is no change in dimension between $\Gamma$ and $\widetilde{\Gamma}$, the only thing we can do is reduce the vertex set. Note that $f_0(\widetilde{\Gamma}) = f_0(\Gamma) - 1$. In order for all the Boolean parts to be attained, we need to have $\widetilde{\Gamma} = \Ast_\Gamma(u)$ (i.e. the entire subcomplex induced by vertex restriction). If we can take some double suspension of $\lk_\Delta(e)$ and stay in $\Delta$, this is the same as double coning of $\Gamma_e$ (which is equivalent to adding 2 new coordinates to the Boolean part). \\
		\end{enumerate}
	\end{proof}

	As mentioned previously \cite{Pbalfveclef}, the Boolean decompositions \[ \Gamma = \{ F \cup G : F \in S, G \in 2^{[d]} \} \] resemble Lefschetz decompositions in the geometric setting of compact K\"ahler manifolds of (complex) dimension $d$. We partially reproduce some of the results and terminology below. \\
	
	\begin{thm} \textbf{(Hard Lefschetz and Lefschetz decomposition, Theorem 3.1 on p. 14 -- 15 of \cite{Pbalfveclef} from Theorem 14.1.1 on p. 237 -- 238 of \cite{Ara}, Theorem on p. 88 of \cite{CMSP}, p. 122 of \cite{GH}, Theorem 6.3 on p. 138 of \cite{Voi}) \\} \label{HLshort}
		\begin{enumerate}
			\item \textbf{(Hard Lefschetz) \\} 

				 Let $(X, \omega)$ be a compact K\"ahler manifold of complex dimension $D$. This induces a homomorphism \[ L : H^k(X) \longrightarrow H^{k + 2}(X) \] given by $[\alpha] \mapsto [\omega \wedge \alpha]$. The map \[ L^r : H^{ D - r }(X) \longrightarrow H^{ D + r }(X) \] is an isomorphism for $0 \le r \le D$. \\

			\item \textbf{(Lefschetz decomposition) \\}
			
			We define by \[ P^i(X) \coloneq \ker( L^{D - i + 1} : H^i(X) \longrightarrow H^{2D - i + 2}(X) ) \] the \textbf{primitive cohomology} of $X$. Equivalently, we have \[ P^{D - r}(X) \coloneq \ker( L^{r + 1} : H^{D - r}(X) \longrightarrow H^{D + r + 2}(X) ). \]
				
			Then, we have that \[ H^m(X) = \bigoplus_{ k = 0 }^{ \lfloor \frac{m}{2} \rfloor } L^k P^{m - 2k}(X). \] This is called the \textbf{Lefschetz decomposition}. \\

		\end{enumerate}
	\end{thm}

	In our particular setting, we will focus on a formal version focusing on the maps of vector spaces involved. \\

	\begin{defn}
		Consider a sequence $L^\cdot$ of maps $L : A^k \longrightarrow A^{k + 2}$ on a sequence of graded vector spaces $A^k$ with $A_d \ne 0$ and $A^k = 0$ for all $k \ge d + 1$. The map $L$ has the \textbf{Lefschetz property} if $L^r$ is an isomorphism of vector spaces for all $0 \le r \le d$ and has a nontrivial kernel for $r \ge d + 1$. \\
	\end{defn}
	
	The maps which we will study this property on are the defined below. \\

	\begin{defn} \label{suspdiv}
		Let $\mathcal{C}$ be a class of flag Cohen--Macaulay simplicial complexes closed under (double) suspensions where those of a given dimension are PL homeomorphic to each other. Given $\Delta \in \mathcal{C}$ of dimension $d - 1$, let $\Gamma$ be a simplicial complex such that $f(\Gamma) = h(\Delta)$. \\

		Fix an even number $d$. We now construct a map $L$ between elements $\Delta \in \mathcal{C}$ graded by $\dim \Delta + 1$. On elements of $\mathcal{C}$ of degree $\le d - 2$, the map $L$ is defined by composing a double suspension with a ``net'' single positive edge subdivision which is either an actual edge subdivision or a collection of edge subdivisions and their inverses ``adding'' to a simple one after splitting the edge subdivisions and contractions into blocks ``adding'' to 1 edge subdivision and 0 on higher degrees. \\
	\end{defn}
	
	We can use this to relate local to global properties of Boolean decompositions of simplicial complexes whose $f$-vectors are equal to $h$-vectors of flag Cohen--Macaulay simplicial complexes. When the simplicial complex $\Delta \in \mathcal{C}$ we take the $h$-vector of is a pseudomanifold without boundary (p. 24 of \cite{St}), the local parts we compare to are Boolean decompositions of simplicial complexes whose $f$-vectors are $h$-vectors of codimension 2 flag spheres inside $\Delta$. \\
	
	\begin{cor} \label{locglobbool}

		Fix an even number $d$. Given a $(d - 1)$-dimensional flag Cohen--Macaulay simplicial complex $\Delta$ and its link $\lk_\Delta(e)$ over an edge $e \in \Delta$, take simplicial complexes $\Gamma$ and $\Gamma_e$ such that $f(\Gamma) = h(\Delta)$ and $f(\Gamma_e) = h(\lk_\Delta(e))$. Suppose that we are in the setting of Definition \ref{suspdiv} and $\Gamma_e$ has a Boolean decomposition for each edge $e \in \Delta$. Then, $\Gamma$ has a Boolean decomposition if and only if the map $L$ has the Lefschetz property.  \\
	\end{cor}
	
	\begin{rem} \label{algdisc}
		Since $\Delta$ is Cohen--Macaulay and flag, $\Gamma$ can be chosen to be $k$-chromatic, where $k$ is the largest index of a nonzero entry of the $h$-vector (Corollary 2.3 on p. 474 of \cite{CCV}, Theorem p. 23 and Section 6.4 on 33 of \cite{BFS}). If $h_d(\Delta) \ne 0$ (e.g. when $\Delta$ is a sphere), $\Gamma$ can be taken to be balanced. In this context, a map to compare $L$ with is the one on $B^\cdot(\Gamma) \coloneq k[\Gamma]/(x_1^2, \ldots, x_n^2)$ defined by adding in missing colors by multiplying by $(1 + \theta_{d - 1})(1 + \theta_d)$, where $\theta_1, \ldots, \theta_d$ is the usual linear system of parameters of a balanced simplicial complex partitioning the vertices of $\Gamma$ into those of a given color (Proposition 4.3 on p. 97 of \cite{St}). We can look at how this interacts with the $x^H$ giving a basis of the Artinian reduction of $k[\Gamma]$ with respect to $\theta_1, \ldots, \theta_d$. Note that the quotienting by the squares of the variables restricts the polynomials in $\theta_i$ to those that are squarefree. As mentioned previously (Section 3 of \cite{Pbalfveclef}), this hints at a sort of anticommutative structure. It would be interesting if we can relate this to Lefschetz-type structures involving exterior algebras from \cite{KR} and \cite{HW} which also involve Boolean components. \\
		
		There may be additional restrictions related to freeness when $\Gamma$ itself is taken to be Cohen--Macaulay (p. 35 of \cite{St}, Theorem 1.5.17 on p. 37 -- 38 of \cite{BH}). When $\Gamma$ is flag, we can take $\Delta$ to be vertex decomposable \cite{CN}. \\

	\end{rem}
	
	\begin{proof}
		The idea is to compare the map $L$ in Definition \ref{suspdiv} with the $f$-vector changes induced by edge subdivisions and contractions between PL homeomorphic flag simplicial complexes of the same dimension in Proposition \ref{subconchange}. Relating this back to a formal version of the Lefschetz decomposition, we formally have maps $L : A^k \longrightarrow A^{k + 2}$ from (collections of) $(k - 1)$-dimensional flag Cohen--Macaulay simplicial complexes in $\mathcal{C}$ to $(k + 1)$-dimensional flag Cohen--Macaulay simplicial complexes in $\mathcal{C}$. Essentially, there are vector space isomorphisms (generated by relevant faces) up to the point we reach degree $d$. \\
		
		For $(d - 1)$-dimensional simplicial complexes in $\mathcal{C}$, we formally have the following:
		
		\begin{itemize} 
			\item $A^d \supset P^d = 2^{[d]} = \ker(L : A^d \longrightarrow A^{d + 2})$ \\
			
			\item $A^{d - 2} \supset P^{d - 2} = 2^{[d - 2]} = \ker(L^2 : A^{d - 2} \longrightarrow A^{d + 2})$ \\
			
			\item $A^2 \supset P^2 = 2^{[2]} = \ker(L^{ \frac{d}{2} } : A^2 \longrightarrow A^{d + 2})$ \\
			
			\item $A^0 \supset P^0 = 2^{[0]} = \ker(L^{ \frac{d}{2} + 1 } : A^0 \longrightarrow A^{d + 2})$ \\
			
			\item In general, we are (formally) considering maps $A^{d - 2k} \supset P^{d - 2k} = 2^{[d - 2k]} = \ker( L^{ k + 1 } : A^{d - 2k} \longrightarrow A^{d + 2})$. \\
		\end{itemize} 
		
		The suspensions and subdivisions are both invertible with inverses given by point removals and contractions. The point is to have an invertible map until we eventually get 0 at the point we go to degree $d + 2$. From Proposition \ref{subconchange}, we see that vertices are added to the initial non-Boolean part under edge subdivisions and removed under edge contractions. Combining these observations leads to the desired statement. 
	\end{proof}

	\color{black}

\end{document}